\documentclass[a4j,10pt,showkeys]{article}

\usepackage{amsmath}    
\usepackage{graphicx}   
\usepackage{geometry}
\usepackage{verbatim}   
\usepackage{color}      
\usepackage{hyperref}   

\usepackage{amssymb}
\usepackage{amsthm}
\usepackage{mathrsfs}

\newtheorem{Def}{Definition}[section]
\newtheorem{Eg}[Def]{Example}
\newtheorem{Prop}[Def]{Proposition}
\newtheorem{Lem}[Def]{Lemma}
\newtheorem{Thm}[Def]{Theorem}

\newtheorem{Asm}[Def]{Assumption}
\theoremstyle{definition}
\newtheorem{Rem}[Def]{Remark}

\newcommand{\rd}{\,\mathrm{d}}
\newcommand{\rmd}{\mathrm{d}}
\newcommand{\1}{{\bf 1}}

\newcommand{\bE}{\mathbb{E}}\newcommand{\bF}{\mathbb{F}}\newcommand{\bN}{\mathbb{N}}\newcommand{\bP}{\mathbb{P}}\newcommand{\bR}{\mathbb{R}}

\newcommand{\cA}{\mathcal{A}}\newcommand{\cF}{\mathcal{F}}

\newcommand{\sL}{\mathscr{L}}

\numberwithin{equation}{section}
\newcommand{\norm}[1]{\left\lVert #1 \right\rVert}

\allowdisplaybreaks

\begin{document}

\title{
Strong rate of convergence for the Euler--Maruyama scheme of SDEs with
unbounded H\"older continuous drift coefficient
}
\author{
Tsukasa Moritoki\footnote{
Department of Mathematics,
Okayama University,
Tsushima-naka,
Kita-ku
Okayama
700-8530,
Japan,
email~:~\texttt{pxyq0l31@s.okayama-u.ac.jp}
}
\quad and \quad
Dai Taguchi\footnote{
Department of Mathematics,
Kansai University,
Suita,
Osaka,
564-8680,
Japan,
email~:~\texttt{taguchi@kansai-u.ac.jp}
}
}
\date{\today}
\maketitle
\begin{abstract}
In this paper, we provide the strong rate of convergence for the Euler--Maruyama scheme for multi-dimensional stochastic differential equations with uniformly locally (unbounded) H\"older continuous drift and multiplicative noise. Our technique is based on It\^o--Tanaka trick (Zvonkin transformation) for unbounded drift. Moreover, in order to apply the stochastic sewing lemma, we use the heat kernel estimate for the density function of the Euler--Maruyama scheme.
%
%
%
%
%
\end{abstract}


\section{Introduction}\label{Sec_1}
Let $X=(X_{t})_{t \in [0,T]}$ be a solution of $d$-dimensional stochastic differential equation (SDE) of the form
\begin{equation}\label{sde_01}
\rd X_t = b(X_t)\rd t + \sigma(X_t)\rd B_t,\ X_0 = x_0 \in \bR^{d},\ t \in [0, T],
\end{equation}
where $B$ is a $d$-dimensional standard Brownian motion on a filtered probability space $(\Omega, \cF, \bF, \bP)$ with a filtration $\bF=(\cF_t)_{t \in [0,T]}$ satisfying the usual conditions.
In this paper, the drift coefficient $b=(b_{1},\ldots,b_{d})^{\top}:\bR^d \to \bR^d$ is assumed to be uniformly locally (unbounded) H\"older continuous with exponent $\alpha \in (0,1]$, that is,
\begin{align*}
\sup_{x \neq y \in \bR^{d},\,|x - y| \leq 1}
\frac{|b(x)-b(y)|}{|x-y|^{\alpha}}
<\infty
\end{align*}
and the diffusion coefficient $\sigma:\bR^d \to \bR^{d \times d}$ is assumed to be bounded, smooth and uniformly elliptic.
Under this conditions, Flandoli, Gubinelli and Priola \cite[Theorem 7]{FlGuPr10} prove the existence of a global flow of diffeomorphisms and pathwise uniqueness for SDE \eqref{sde_01}.
Their technique for the proof is based on a transformation of the drift coefficient which is called It\^o--Tanaka trick (or Zvonkin transformation \cite{Ve81,Zv74}).
More precisely, they prove that for sufficiently large $\lambda>0$, the equation
\begin{align*}
\lambda u_{\lambda}
-
\sL u_{\lambda}
=
b
\end{align*}
admits a unique classical solution $u_{\lambda}=(u_{\lambda,1},\ldots,u_{\lambda,d})^{\top}$.
Here the operator $\sL$ is the infinitesimal generator of $X$, that is,
\begin{align*}
(\sL g)(x)
:=
\frac{1}{2}\mathrm{Tr}\big(a(x)(\nabla^{2}g)(x)\big)
+
\langle b(x),(\nabla g)(x) \rangle,~a(x):=\sigma(x) \sigma(x)^{\top}.
\end{align*}
Then by using It\^o's formula, we have for $i=1,\ldots,d$,
\begin{align*}
\int_{0}^{t}
b_{i}(X_{s})
\rd s
=
u_{\lambda,i}(x_{0})
-
u_{\lambda,i}(X_{t})
+
\lambda
\int_{0}^{t}
u_{\lambda,i}(X_{s})
\rd s
+
\int_{0}^{t}
\langle (\nabla u_{\lambda,i})(X_{s}),\sigma(X_{s}) \rd B_{s} \rangle,
\end{align*}
and this transformation shows that we can use the regularity of the solution $u_{\lambda,i}$ to prove (for example) pathwise uniqueness and to provide the rate of convergence for the numerical scheme.

In this paper, we concern with the strong rate of convergence for the Euler-Maruyama scheme $X^n=(X_t^n)_{t \in [0,T]}$ defined by
\begin{equation}\label{em_01}
\rd X_t^n = b(X_{\kappa_n(t)}^n)\rd t + \sigma(X_{\kappa_n(t)}^n)\rd B_t,\ X_0^n = x_0, \ t \in [0, T],
\end{equation}
under the above regularity conditions for the coefficients, in particular uniformly locally (unbounded) H\"older continuous drift coefficient.
Here $\kappa_n$ is defined by $\kappa_n(t) := t_k := kT/n, t \in [t_k, t_{k + 1})$.
It is well-known that if the coefficients $b$ and $\sigma$ are Lipschitz continuous, then the Euler--Maruyama scheme has the strong rate of convergence $1/2$ (e.g. \cite{KlPl92}), that is, for any $p \geq 1$,
\begin{align*}
\bE\Big[\sup_{t \in [0, T]}|X_{t}-X_{t}^{n}|^{p}\Big]^{1/p}
\lesssim
n^{-1/2}.
\end{align*}
In the case of non-Lipschitz coefficients, Kaneko and Nakao \cite[Theorem D]{KaNa88} prove that if pathwise uniqueness holds for SDE \eqref{sde_01} with continuous and linear growth coefficients, then the Euler--Maruyama scheme converges in $L^{2}$, but they do not study the rate of convergence.
For the results concerning wtih the strong rate of convergence, if the drift coefficient is bounded H\"older continuous with exponent $\alpha \in (0,1)$ and the diffusion coefficient is the identity matrix, Menoukeu Pamen and Taguchi \cite{MeTa17} prove that the Euler--Maruyama scheme has the strong rate of convergence $\alpha/2$, that is, for any $p \geq 1$,
\begin{align*}
\bE[|X_{T}-X_{T}^{n}|^{p}]^{1/p}
\lesssim
n^{-\alpha/2}.
\end{align*}
The technique for the proof is based on It\^o--Tanaka trick to remove the H\"older drift.
In recent years, this result is extended in many papers (e.g. \cite{BaDiPa23,BaHuYu19,BaHuZh22,BaWu25,BuDaGe21,DaGe20,DaGeLe23,GeLaLi25,LeLi25,NeSz21,NgTa17}).
In particular, Butkovsky, Dareiotis and Gerencs\'er \cite{BuDaGe21} improve the strong rate of convergence, and prove that for any $p \geq 1$,
\begin{align*}
\bE\Big[\sup_{t \in [0, T]}|X_{t}-X_{t}^{n}|^{p}\Big]^{1/p}
\lesssim
\left\{\begin{array}{ll}
\displaystyle
n^{-1/2+\varepsilon},
&\text{ if } \sigma \in C^{2}_{b} \text{ and uniformly elliptic},\\
\displaystyle
n^{-(1+\alpha)/2+\varepsilon},
&\text{ if } \sigma \text{ is the identity matrix},
\end{array}\right.
\end{align*}
The key idea of the proof is to use the stochastic sewing lemma introduced by L\'e \cite{Le20}.
The stochastic sewing lemma is a stochastic generalization of the deterministic sewing lemma \cite{Fede06,Gu04}, which gives a sufficient condition for the convergence of random Riemann sums.
Moreover, for uniformly locally (unbounded) H\"older continuous drift and $\sigma=I_{d}$, Babi, Dieye and Menoukeu Pamen \cite{BaDiPa23} provide the strong rate of convergence for sufficiently small $T>0$,
They use It\^o--Tanaka trick \cite{FlGuPr10}.

In this paper, we provide the strong rate of convergence for the Euler--Maruyama scheme with uniformly locally H\"older continuous drift and multiplicative noise.
More precisely, we prove that if we additionally suppose that the drift is of sub-linear growth, then for any $T>0$ and $p \geq 1$, it holds that
\begin{align*}
\bE[|X_{T}-X_{T}^{n}|^{p}]^{1/p}
\lesssim
\left\{\begin{array}{ll}
\displaystyle
n^{-1/2},
&\text{ if } \sigma \in C^{3}_{b} \text{ and uniformly elliptic},\\
\displaystyle
n^{-(1+\alpha)/2+\varepsilon},
&\text{ if } \sigma \text{ is the identity matrix}.
\end{array}\right.
\end{align*}
The idea of the proof for our main results is to use the stochastic sewing lemma and It\^o--Tanaka trick.
In order to apply the stochastic sewing lemma, we use the heat kernel estimate for the density function of the Euler--Maruyama scheme (without drift) and its derivatives proved by Konakov and Mammen \cite{KoMa02}.
More precisely, they prove that by using the parametrix method, if the coefficients are sufficiently smooth, then the density function $p^{n}_{t_{k}}(x,\cdot)$ of $X^{n}_{t_{k}}$ with $X^{n}_{0}=x$ satisfies the following heat kernel estimate: for some multiindex $\ell,m$, there exists $C_{+},c_{+}>0$ such that for any $y \in \bR^{d}$,
\begin{align*}
|\partial_{y}^{\ell}\partial_{x}^{m}p^{n}_{t_{k}}(x,y)|
\leq
C_{+}
t_{k}^{-(|\ell|+|m|)}
\frac{e^{-\frac{|y-x|^{2}}{2c_{+}t_{k}}}}{(2\pi c_{+}t_{k})^{d/2}}.
\end{align*}
Then we use Girsanov transformation to add the drift term for the Euler--Maruyama scheme.
In order to confirm Novikov condition (the drift can be unbounded, but it is of sub-linear growth), we use Burkholder--Davis--Gundy's inequality with a sharp constant (e.g. \cite{BaYo82,Re08}).
Finally, we use It\^o--Tanaka trick for unbounded drift coefficient \cite{FlGuPr10} in order to transform the drift coefficient.


\subsection{Notations}
Throughout this paper, we will use the following notations.
For $x,y \in \bR$,  $x \lesssim y$ means that there is a constant $c > 0$ such that $x \le cy$.
For $x \in \bR^{d}$, $|x|$ denotes Euclidean norm.
For a matrix $A$, $|A|$ denotes the Frobenius norm.
For a matrix valued function $a$, we define $\|a\|_{\infty} := \sup_{x} |a(x)|$.
$I_d$ denotes $d$-dimensional identity matrix.
For a random variable $X$ on a filtered probability space $(\Omega, \cF, \bF, \bP)$ with a filtration $\bF=(\cF_t)_{t \in [0,T]}$, we set $\|X \|_{L^{p}}:=\bE[|X|^{p}]^{1/p}$ and $\bE^t[X]:=\bE[X\,|\,\cF_{t}]$, $t \in [0,T]$.
For a $d$-dimensional symmetric positive definite matrix $A$,  we define
\begin{equation*}
g_{A}(x)
:=
\frac{1}{(2\pi)^{d/2}(\det A)^{1/2}}
\exp\Big(-\frac{\langle A^{-1} x, x \rangle}{2}\Big),~x \in \bR^d,
\end{equation*}
and $g_c(x) := g_{cI_d}(x)$ for $c > 0$.

\section{Main results}\label{Sec_2}

We first introduce some space of functions.
Let $d,k,n \in \bN$ and $\alpha \in (0,1]$.

A measurable function $f:\bR^{d} \to \bR$ is called \textit{linear growth} if $\|(1 + |\cdot|)^{-1}f(\cdot) \|_{\infty}<\infty$.
The function $f$ is called \textit{sub-linear growth} if $f$ is bounded on any compact sets and satisfies $\lim_{|x| \to \infty} |f(x)| / |x| = 0$.
In the latter case, we have the following growth condition: for any $\delta > 0$, there exists $L_{\delta} > 0$ such that for any $x \in \bR^{d}$, $|f(x)| \leq \delta |x| + L_{\delta}$.

$C^{\alpha}(\bR^{d};\bR^{k})$ denotes the set of all functions from $\bR^{d}$ to $\bR^{k}$ which are uniformly locally $\alpha$-H\"older continuous functions, that is,
\begin{align*}
[f]_{\alpha}
:=
\sup_{x \neq y \in \bR^{d},\,|x - y| \leq 1}
\frac{|f(x) - f(y)|}{|x - y|^\alpha}
<
\infty.
\end{align*}
Note that if $f \in C^{\alpha}(\bR^{d};\bR^{k})$, then it is of linear growth.
Indeed, let $x\in \bR^{d}$ and $m=\lfloor |x| \rfloor+1$, where $\lfloor a \rfloor$ denotes the floor function.
Then it holds that
\begin{align}
|f(x)|
&\leq
|f(0)|
+
\sum_{j = 0}^{m - 1}
\Big|f\Big(\frac{j}{m}x\Big) - f\Big(\frac{j+1}{m}x\Big)\Big|
\notag
\leq
|f(0)|
+
[f]_\alpha
\sum_{j = 0}^{m - 1} \Big(\frac{|x|}{m}\Big)^\alpha
\\&=
|f(0)|
+
[f]_{\alpha}m^{1-\alpha}|x|^{\alpha}
\leq
|f(0)|
+
[f]_\alpha (1 + |x|) .
\label{eq_linear_1}
\end{align}
Thus $f$ is of linear growth.
The set $C^{\alpha}(\bR^{d};\bR^{k})$ is a Banach space with the norm
\begin{align*}
\|f\|_{\alpha}:=\|(1 + |\cdot|)^{-1}f(\cdot) \|_{\infty}+[f]_{\alpha}.
\end{align*}
Moreover, for $\alpha,\beta \in (0,1]$ with $\alpha<\beta$, it holds that $C^{\alpha}(\bR^{d};\bR^{k}) \subset C^{\beta}(\bR^{d};\bR^{k})$.
In particular, if $f\in C^{1}(\bR^{d};\bR^{k})$, then $f$ is globally Lipschitz continuous, that is,
\begin{align*}
[f]_{\mathrm{Lip}}:=\sup_{x \neq y}\frac{|f(x)-f(y)|}{|x-y|}<\infty.
\end{align*}
Indeed, for $x,y \in \bR^{d}$ with $|x-y|>1$, we set $m=\lfloor |x-y| \rfloor+1$.
Then we have
\begin{align*}
|f(x)-f(y)|
&\leq
\sum_{j=0}^{m-1}
\Big|
f\Big(x+\frac{j}{m}(y-x)\Big)
-
f\Big(x+\frac{j+1}{m}(y-x)\Big)
\Big|
\leq
[f]_{1}
\sum_{j=0}^{m-1}
\frac{|x-y|}{m}
=
[f]_{1}|x-y|.
\end{align*}

$C_{b}^{n}(\bR^{d};\bR^{k})$ is the space of all functions from $\bR^{d}$ to $\bR^{k}$ having bounded derivatives up to the order $n$, that is,
\begin{align*}
\|f\|_{C_{b}^{n}}:= \sum_{i = 0}^{n} \|\nabla^{i}f\|_{\infty} < \infty.
\end{align*}
Moreover, $C^{n + \alpha}(\bR^{d}:\bR^{k})$ denotes the set of all functions from $\bR^{d}$ to $\bR^{k}$ satysfying
\begin{align*}
\|f\|_{n + \alpha}:=\| (1 + |\cdot|)^{-1}f(\cdot)\|_\infty +[\nabla^n f]_{\alpha}+\|\nabla f\|_{C_{b}^{n-1}} < \infty.
\end{align*}

We have several examples for the class $C^{\alpha}(\bR^{d};\bR)$.

\begin{Eg}\label{Eg_1}
\begin{itemize}
\item[(i)]
Let $\alpha,\beta \in (0,1)$.
Define $f(x):=|x|^{\alpha}+|x|^{\beta}$, $x \in \bR^{d}$.
Then $f \in C^{\gamma}(\bR^{d};\bR)$, where $\gamma:=\alpha\wedge \beta$ and $f$ is of sub-linear growth.
However, if $\alpha< \beta$ (or $\alpha>\beta$), then $f$ is not globally H\"older continuous for any exponent $\rho \in (0,1]$.
Indeed, for any $x \in \bR^{d}$, it holds that
\begin{align*}
\frac{|f(x)-f(0)|}{|x-0|^{\rho}}
=
|x|^{\alpha-\rho}+|x|^{\beta-\rho}
\to \infty,
\end{align*}
as $|x| \to \infty$ if $\alpha>\rho$ or $\alpha \leq \rho<\beta$, and as $|x| \to 0$ if $\alpha<\beta \leq \rho$.

\item[(ii)]
Let $\alpha \in (0,1)$ and $f(x):=|x|^{\alpha} \log(2+|x|)$, $x \in \bR^{d}$.
Then $f \in C^{\alpha}(\bR^{d};\bR)$ and $f$ is of sub-linear growth.
Indeed, the map $\bR^{d} \ni x \mapsto |x|\in \bR$ is globally Lipschitz continuous, thus we can assume $d=1$.
If $|x|, |y|\geq 1$, then $f$ is smooth and $f'$ is bounded, hence for  $|x-y|\leq 1$,
\begin{align*}
|f(x)-f(y)|
\leq
\|f'\|_{\infty}|x-y|
\leq
\|f'\|_{\infty}|x-y|^{\alpha}.
\end{align*}
If $|x|<1$ and $|x-y|\leq 1$, then we have
\begin{align*}
|f(x)-f(y)|
&\leq
\log(2+|x|)
\big||x|^{\alpha}-|y|^{\alpha}\big|
+
|y|^{\alpha}
\big|\log(2+|x|)-\log(2+|y|)\big|
\\&\leq
(\log 3) |x-y|^{\alpha}
+
\frac{2^{\alpha}}{2+|x|\wedge |y|}|x-y|
\leq
(\log 3+2^{\alpha-1})|x-y|^{\alpha}.
\end{align*}
Hence $f \in C^{\alpha}(\bR;\bR)$.
However, $f$ is not globally H\"older continuous for any exponent $\rho \in (0,1]$.
Indeed, for any $x \in \bR$ and  $\rho\in (0,1]$, we have
\begin{align*}
\frac{|f(x)-f(0)|}{|x-0|^{\rho}}
=
|x|^{\alpha-\rho}\log(2+|x|)
\to \infty,
\end{align*}
as $|x| \to \infty$ if $\alpha \geq \rho$, and as $|x| \to 0$ if $\alpha<\rho$.
\end{itemize}

\end{Eg}

We need the following assumptions for the coefficients of SDE \eqref{sde_01}.

\begin{Asm}\label{asm_01}
We assume that the drift coefficient $b:\bR^{d} \to \bR^{d}$ and the diffusion coefficient $\sigma$ satisfy the following conditions.
\begin{itemize}
\item[(i)]
$b \in C^{\alpha}(\bR^d; \bR^d)$ with $\alpha \in (0,1]$.
\item[(ii)]
$b$ is of sub-linear growth.
\item[(iii)]
$\sigma \in C_{b}^{3}(\bR^d; \bR^{d \times d})$.
\item[(iv)]
There exists $\lambda \geq 1$ such that for any $x,\xi \in \bR^{d}$
\begin{equation*}
\lambda^{-1}|\xi|^{2}
\leq
\langle a(x)\xi,\xi \rangle
\leq
\lambda|\xi|^{2},~a(x)=\sigma(x) \sigma(x)^{\top}.
\end{equation*}
\end{itemize}
\end{Asm}

Under Assumption \ref{asm_01} (i), (iii) and (iv), there exists a unique strong solution to SDE \eqref{sde_01}  (see, \cite[Theorem 7]{FlGuPr10}).

We provide the following strong rate of convergence for the Euler--Maruyama scheme \eqref{em_01}.

\begin{Thm}\label{main_thm00}
We suppose that Assumption \ref{asm_01} holds.
Then for any $p \geq 1$ and $\varepsilon \in (0,(1+\alpha)/2)$, there exist constants $C_{p}>0$ and $C_{p,\varepsilon}>0$ such that
\begin{align*}
\sup_{t \in [0, T]}
\bE[|X_{t}-X_{t}^{n}|^{p}]^{1/p}
\leq
\left\{\begin{array}{ll}
\displaystyle
C_{p}n^{-1/2},
&\text{ if } \sigma \text{ is not a constant matrix,}
\\
\displaystyle
C_{p,\varepsilon}n^{-(1+\alpha)/2+\varepsilon},
&\text{ if } \sigma \text{ is a constant matrix.}
\end{array}\right.
\end{align*}
\end{Thm}

The main forcus of Theorem \ref{main_thm00} is on the case $\alpha \in (0,1)$, but the case $\alpha=1$ (that is, $b$ is globally Lipschitz and is of sub-linear growth) and a constant matrix $\sigma$ is also new result.
If we remove Assumption \ref{asm_01} (ii), then we can prove the same convergence rate for sufficiently small $T>0$ (see Remark \ref{Rem_linear}).


\begin{Rem}
Butkovsky, Dareiotis and Gerencs\'er \cite{BuDaGe21} prove that if the drift is bounded globally $\alpha$-H\"older continuous with $\alpha \in (0,1]$, then for any $p \geq 1$,
\begin{align*}
\bE\Big[\sup_{t \in [0, T]}|X_{t}-X_{t}^{n}|^{p}\Big]^{1/p}
\lesssim
\left\{\begin{array}{ll}
\displaystyle
n^{-1/2+\varepsilon},
&\text{ if } \sigma \in C^{2}_{b} \text{ and uniformly elliptic},\\
\displaystyle
n^{-(1+\alpha)/2+\varepsilon},
&\text{ if } \sigma \text{ is the identity matrix},
\end{array}\right.
\end{align*}
(see \cite[Theorem 2.1 and Theorem 2.7]{BuDaGe21}).
Babi, Dieye and Menoukeu Pamen \cite{BaDiPa23} prove that if $b \in C^{\alpha}(\bR^{d};\bR^{d})$ with $\alpha \in (0,1)$ and $\sigma=I_{d}$, then for any $\varepsilon \in (0,1/2)$, there exsists $T_{\varepsilon}>0$ such that
\begin{align*}
\sup_{t \in [0, T_{\varepsilon}]}
\bE[|X_{t}-X_{t}^{n}|^{2}]^{1/2}
\lesssim
n^{-1/2+\varepsilon},
\end{align*}
(see \cite[Theorem 1.1]{BaDiPa23}).
Therefore, Theorem \ref{main_thm00} is a generalization of the above results.

Moreover, Ellinger, M{\"u}ller-Gronbach and Yaroslavtseva \cite{ElMuYa25} prove that if $d=1$ and $\sigma=1$, then there exists a bounded $\alpha$-H\"older continuous function $b$ such that
\begin{align*}
\inf_{g:\bR^{n} \to \bR \text{ measurable }}
\bE[|X_{T}-g(B_{t_{1}},\ldots,B_{t_{n}})|]
\gtrsim
n^{-(1+\alpha)/2},
\end{align*}
(see \cite[Theorem 1]{ElMuYa25}).
Therefore, the convergence rate in Theorem \ref{main_thm00} is almost optimal.

\end{Rem}

\section{Proof of main theorem}\label{Sec_3}

In this section, we provide the proof of Theorem \ref{main_thm00}.

\subsection{Some auxiliary estimates}\label{Sec_3_1}

In this section, we provide some auxiliary estimates for the Euler--Maruyama scheme \eqref{em_01}.
The following lemma can be shown by using standard stochastic calculus.

\begin{Lem}\label{lem_dif_em}
We suppose that the coefficients $b$ and $\sigma$ are of linear growth.
Then for any $p \geq 1$, there exists a constant $C_{p} > 0$ such that for any $0 \leq s \leq t \leq T$,
\begin{align*}
\sup_{n \in \bN}
\sup_{t \in [0, T]}\bE[|X_t^n|^p]
\leq
C_{p}
\text{ and }
\sup_{n \in \bN}
\bE[|X_{t}^{n} - X_{s}^{n}|^p]
\leq
C_{p} (t-s)^{p/2}.
\end{align*}
\end{Lem}

Next, we provide the following quadrature estimate.

\begin{Lem}\label{lem_p_em}
We suppose that the coefficient $b$ and $\sigma$ are of linear growth.
Let $f:\bR^{d} \to \bR^{d}$ be a linear growth measurable function and $g:\bR^{d}\times \bR^{d} \to \bR$ be a bounded measurable function.
Then for any $p \geq 1$, there exists a constant $C_{p} > 0$ such that for any $0 \leq s \leq t \leq T$ and $n \in \bN$,
\begin{align*}
&
\Big\|
\int_{s}^{t}
\Big\{f(X_{r}^{n})-f(X_{\kappa_n(r)}^{n})\Big\}
g(X_{r}^{n},X_{s}^{n})
\rd r
\Big\|_{L^{p}}
\leq
C_{p}
\|(1+|\cdot|)^{-1}f(\cdot)\|_{\infty}
\|g\|_{\infty}
(t - s).
\end{align*}
\end{Lem}
\begin{proof}
By using Lemma \ref{lem_dif_em}, we have
\begin{align*}
&\Big\|
\int_{s}^{t}
\Big\{f(X_{r}^{n})-f(X_{\kappa_n(r)}^{n})\Big\}
g(X_{r}^{n},X_{s}^{n})
\rd r
\Big\|_{L^{p}}
\\&\leq
\int_s^t
\Big\|\{f(X_r^n) - f(X_{\kappa_n(r)}^n)\}g(X_{r}^{n},X_{s}^{n})\Big\|_{L^p}
\rd r\\
&\leq
\|(1+|\cdot|)^{-1}f(\cdot)\|_{\infty}
\|g\|_{\infty}
\int_s^t
\Big\{2+\|X_r^n\|_{L^{p}} + \|X_{\kappa_n(r)}^n\|_{L^{p}}\Big\}
\rd r\\
&\lesssim
\|(1+|\cdot|)^{-1}f(\cdot)\|_{\infty}
\|g\|_{\infty}
(t - s).
\end{align*}
This completes the proof.
\end{proof}

\subsection{Quadrature estimates}\label{Sec_3_2}

The goal of this section is to prove the following quadrature estimate:
suppose that Assumption \ref{asm_01} holds and let $f \in C^{\alpha}(\bR^{d};\bR)$ with $\alpha \in (0,1]$ and $g:\bR^{d} \to \bR$ be a bounded global Lipschitz continuous function, then it holds that
\begin{align*}
\sup_{t \in [0, T]}
\Big\|
\int_{0}^{t}
\Big\{f(X_{r}^{n})-f(X_{\kappa_n(r)}^{n})\Big\}
g(X_{r}^{n})
\rd r
\Big\|_{L^{p}}
\lesssim
([f]_{\alpha} \vee |f(0)|)(\|g\|_\infty + [g]_{\mathrm{Lip}}) n^{-\frac{1 + \alpha}{2} + \varepsilon},
\end{align*}
(see Proposition \ref{lem_quadratic_em}).
We first prove the above estimate for the Euler-Maruyama scheme without drift, and then we extend it to $X^n$ by  using Maruyama--Girsanov theorem.

Let $Y^{n}=(Y^{n}_{t})_{t \in [0,T]}$ be the Euler-Maruyama scheme without drift coefficient, that is,
\begin{equation*}
\rd Y_t^n = \sigma(Y_{\kappa_n(t)}^n)\rd B_t,~Y_{0}^{n}=x_{0},~t \in [0,T].
\end{equation*}
Then we obtain the following quadrature estimate for $Y^{n}$.

\begin{Prop}\label{lem_Lp}
Let $f \in C^{\alpha}(\bR^{d};\bR)$ with $\alpha \in (0,1]$ and $g:\bR^{d} \to \bR$ be a bounded and global Lipschitz continuous function.
We suppose that the diffusion coefficient $\sigma$ satisfies Assumption \ref{asm_01} (iii) and (iv).
Then for any $p \geq 1$ and $\varepsilon \in (0,(1+\alpha)/2)$, there exists a constant $C_{p,\varepsilon} > 0$ such that for any $t \in [0, T]$ and $n \in \bN$,
\begin{align}\label{ineq_prop}
\sup_{t \in [0,T]}
\Big\|
\int_{0}^{t}
\Big\{f(Y_{r}^{n})-f(Y_{\kappa_n(r)}^{n})\Big\}
g(Y_{r}^{n})
\rd r
\Big\|_{L^{p}}
\leq
C_{p,\varepsilon}
([f]_{\alpha} \vee |f(0)|)(\|g\|_\infty + [g]_{\mathrm{Lip}}) n^{-\frac{1 + \alpha}{2} + \varepsilon}.
\end{align}
\end{Prop}

We prove this proposition by using the following stochastic sewing lemma introduced by L\'e in \cite{Le20} similar to the proof of \cite[Lemma 6.1]{Ho24}.

\begin{Lem}[Stochastic sewing lemma, \text{\cite[Theorem 2.3]{Le20}}]\label{lem_stochastic_sewing}
Let $m \geq 2$ and $(A_{s, t})_{0 \leq s \leq t \leq T}$ be a two parameter stochastic process with values in $\bR^{d}$ which is $L^{m}$-integrable and adapted to the filteration $\bF=(\cF_t)_{t \in [0,T]}$, that is, for each $0\leq s \leq t \leq T$, $\bE[|A_{s,t}|^{m}]<\infty$ and $A_{s,t}$ is $\cF_{t}$-measurable.
Assume that there exist constants $\Gamma_{1}, \Gamma_{2} \geq 0$, $\varepsilon_{1}, \varepsilon_{2} > 0$ such that for every $0 \leq s\leq u \leq t \leq T$,
\begin{align}
\|A_{s, t}\|_{L^m}
&\leq
\Gamma_{1}
|t - s|^{1/2 + \varepsilon_1},
\label{sew_1}
\\
\|\bE^{s}[A_{s,t}-A_{s,u}-A_{u,t}]\|_{L^m}
&\leq
\Gamma_2|t - s|^{1 + \varepsilon_2}.
\label{sew_2}
\end{align}
Then there exists a unique stochastic process $\cA=(\cA_t)_{t\in [0,T]}$ with values in $\bR^d$
satisfying the following properties.
\begin{itemize}
\item[(i)]
$\cA_0 = 0$ and $\cA_t$ is $\cF_t$-measurable and $L^m$-integrable.
\item[(ii)]
There are constants $\Gamma_3, \Gamma_4 \ge 0$ such that for every $0 \le s < t \le T$,
\begin{align}
\|\cA_t - \cA_s - A_{s, t}\|_{L^m}
&\leq
\Gamma_{3}
|t - s|^{1/2 + \varepsilon_1}
+
\Gamma_{4}|t - s|^{1 + \varepsilon_2},
\label{sew_3}\\
\|\bE^{s}[\cA_t - \cA_s - A_{s, t}]\|_{L^m} 
&\leq
\Gamma_4|t - s|^{1 + \varepsilon_2}.
\label{sew_4}
\end{align}
\end{itemize}
Moreover, there exists a constant $K>0$ depending only on $\varepsilon_{1}, \varepsilon_{2}$ and $d$ such that for any $0 \leq s \leq t \le T$,
\begin{align}\label{sew_5}
\|\cA_t - \cA_s\|_{L^m}
\leq
K \Gamma_1|t - s|^{1/2 + \varepsilon_1}
+K \Gamma_2|t - s|^{1 + \varepsilon_2}.
\end{align}
\end{Lem}

We also need the following lemma in order to deal with the uniformly local H\"older continuity of $f$.

\begin{Lem}\label{lem_tail}
Let $f \in C^{\alpha}(\bR^{d};\bR)$ with $\alpha \in (0,1]$.
We suppose that the diffusion coefficient $\sigma$ satisfies Assumption \ref{asm_01} (iv).
Then for any $p \geq 1$, there exists $C_{p}>0$ such that for any $0\leq s \leq r \leq T$,
\begin{align*}
\Big\|
\bE^{s}\Big[
\Big\{f(Y_r^n) - f(Y_{\kappa_n(r)}^n)\Big\}
1_{\{|Y_r^n - Y_{\kappa_n(r)}^n| > 1\}}
\Big]
\Big\|_{L^p}
\leq
C_{p}([f]_{\alpha} \vee |f(0)|)
\exp\Big(-\frac{n}{8 p \lambda T}\Big).
\end{align*}
\end{Lem}
\begin{proof}
Since $f$ satisfies \eqref{eq_linear_1}, by using Jensen's inequality, H\"older's inequality and Lemma \ref{lem_dif_em}, we have
\begin{align*}
&\bE\Big[
\Big|\bE^{s}\Big[
\Big\{f(Y_r^n) - f(Y_{\kappa_n(r)}^n)\Big\}
1_{\{|Y_r^n - Y_{\kappa_n(r)}^n| > 1\}}
\Big]
\Big|^{p}
\Big]
\leq
\bE\Big[
\Big|f(Y_r^n) - f(Y_{\kappa_n(r)}^n)\Big|^{p}
1_{\{|Y_r^n - Y_{\kappa_n(r)}^n| > 1\}}
\Big]
\\&\leq
\bE\Big[
\Big|f(Y_r^n) - f(Y_{\kappa_n(r)}^n)\Big|^{2p}
\Big]^{1/2}
\bP(|Y_r^n - Y_{\kappa_n(r)}^n| > 1)^{1/2}
\\&\lesssim
([f]_{\alpha} \vee |f(0)|)^{p}
\bP(|Y_r^n - Y_{\kappa_n(r)}^n| > 1)^{1/2}.
\end{align*}
Let $t_{k}=\kappa_{n}(r)$.
Then by the definition of the Euler--Maruyama scheme $Y^{n}$, it holds that
\begin{align*}
Y_r^n-Y_{t_{k}}^n
=
\sigma(Y_{t_{k}}^n)(B_{r}-B_{t_{k}}).
\end{align*}
Since $\sigma(Y_{t_{k}}^n)$ and $B_{r}-B_{t_{k}}$ are independent, we have
\begin{align*}
\bP(|Y_r^n - Y_{\kappa_n(r)}^n| > 1)
&=
\bE\Big[
\bE^{t_{k}}
\Big[
\1_{\{|\sigma(Y_{t_{k}}^n)(B_{r}-B_{t_{k}})|>1\}}
\Big]
\Big]
=
\bE\Big[
\int_{|y|>1}
g_{a(Y_{t_{k}}^n)(r-t_{k})}(y)
\rd y
\Big]
\\&\leq
\sup_{z \in \bR^{d}}
\int_{|y| > 1} g_{a(z)(r-t_{k})}(y)\rd y 
\leq
\lambda^{d} \int_{|y| > 1} g_{\lambda(r-t_{k})}(y)\rd y
\\&\leq
\frac{\lambda^{d}}{(2\pi)^{d/2}}
\int_{\sqrt{\lambda (r-t_{k})}|u|>1}
\exp\Big(-\frac{|u|^2}{2}\Big)
\rd u
\\&\leq
\frac{\lambda^{d}}{(2\pi)^{d/2}}
\exp\Big(-\frac{1}{4\lambda (r-t_{k})}\Big)
\int_{\bR^{d}}
\exp\Big(-\frac{|u|^{2}}{4}\Big)
\rd u
\\&\lesssim
\exp\Big(-\frac{n}{4\lambda T}\Big).
\end{align*}
This concludes the proof.
\end{proof}

Now we are ready to prove Proposition \ref{lem_Lp}.

\begin{proof}[Proof of Proposition \ref{lem_Lp}]
It suffices to prove this proposition when the test function $f$ is sufficiently smooth.
Indeed, let $\rho(x):=C^{-1}\exp(-\frac{1}{1-|x|^{2}})$, $C=\int_{|x|<1} \exp(-\frac{1}{1-|x|^{2}}) \rd x$, define $\rho_{m}(x):=m^{d}\rho(mx)$ and $f_{m}(x):=(\rho_{m}*f)(x)$, $x \in \bR^d$, $m \in \bN$.
Then it holds that $\lim_{m \to \infty}f_{m}(x)=f(x)$, $x \in \bR^{d}$ and for $x,y \in \bR^{d}$ with $|x-y| \leq 1$,
\begin{align*}
|f_{m}(x)-f_{m}(y)|
=
\Big|
\int_{\bR^{d}}
\rho_{m}(z)
\Big\{f(x-z)-f(y-z)\Big\}\rd z
\Big|
\leq
[f]_{\alpha}
\int_{\bR^d}
\rho_{m}(z)
|x-y|^{\alpha}
\rd z
=[f]_{\alpha}
|x-y|^{\alpha}.
\end{align*}
This implies $[f_{m}]_{\alpha} \leq [f]_{\alpha}$.
Hence if the estiamte \eqref{ineq_prop} holds with $f_{m}$ in place of $f$, then by Fatou's lemma, we obtain the estimate for $f$.

Now we prove \eqref{ineq_prop} for $p \geq 2$ and $\varepsilon \in (0, 1/2)$ by using the stochastic sewing lemma (Lemma \ref{lem_stochastic_sewing}).
We define a two parameter stochastic process $(A_{s, t})_{0 \leq s \leq  t \leq T}$ by
\begin{align*}
A_{s, t}
:=
\bE^{s}
\Big[
\int_s^t
\Big\{f(Y_r^n)-f(Y_{\kappa_n(r)}^n)\Big\}
g(Y_s^n)
\rd r
\Big],~0 \leq s \leq  t \leq T.
\end{align*}
Then we prove that $(A_{s, t})_{0 \leq s \leq  t \leq T}$ satisfies the conditions of Lemma \ref{lem_stochastic_sewing}.

We first prove that $(A_{s, t})_{0 \leq s \leq  t \leq T}$ satisfies the condition \eqref{sew_1}.
Let $0 \leq s \leq  t \leq T$ be fixed, and then there exists $k \in \{0, 1, \ldots, n - 1\}$ such that $\kappa_n(s) = t_k$.
We first assume $t > t_{k + 3}$.
Then by using Lemma \ref{lem_tail}, we have
\begin{align}\label{A_st_1}
\|A_{s, t}\|_{L^p}
&\leq
\|g\|_\infty
\int_s^{t}
\Big\{
\Big\|
\bE^{s}\Big[
\Big\{f(Y_r^n) - f(Y_{\kappa_n(r)}^n)\Big\}
1_{\{|Y_r^n - Y_{\kappa_n(r)}^n| \leq 1\}}
\Big]
\Big\|_{L^p}
\notag
\\&\hspace{3cm}+
\Big\|
\bE^{s}\Big[
\Big\{f(Y_r^n) - f(Y_{\kappa_n(r)}^n)\Big\}
1_{\{|Y_r^n - Y_{\kappa_n(r)}^n| > 1\}}
\Big]
\Big\|_{L^p}
\Big\}
\rd r
\notag
\\&\lesssim
\|g\|_\infty
\Big(
\int_{s}^{t_{k+2}}+\int_{t_{k+2}}^{t}
\Big)
\Big\|
\bE^{s}\Big[
\Big\{f(Y_r^n) - f(Y_{\kappa_n(r)}^n)\Big\}
1_{\{|Y_r^n - Y_{\kappa_n(r)}^n| \leq 1\}}
\Big]
\Big\|_{L^p}
\rd r
\notag
\\&\hspace{3cm}+
\|g\|_\infty([f]_{\alpha} \vee |f(0)|)
\exp\Big(-\frac{n}{8 p \lambda T}\Big)
|t-s|.
\end{align}
By using the uniformly local H\"older continuity of $f$ and Lemma \ref{lem_dif_em}, we have
\begin{align}\label{J_1}
\int_{s}^{t_{k+2}}
\Big\|
\bE^{s}\Big[
\Big\{f(Y_r^n) - f(Y_{\kappa_n(r)}^n)\Big\}
1_{\{|Y_r^n - Y_{\kappa_n(r)}^n| \leq 1\}}
\Big]
\Big\|_{L^p}
\rd r
&\leq
[f]_{\alpha}
\int_{s}^{t_{k+2}}
\bE[|Y_r^n-Y_{\kappa_n(r)}^{n}|^{p\alpha}]^{1/p}
\rd r\notag
\\&\lesssim
[f]_{\alpha}
n^{-\alpha/2}
(t_{k + 2} - s) \notag
\\&\leq
[f]_{\alpha}
n^{-\alpha/2}
(t_{k + 2} - t_{k})^{1/2-\varepsilon}
|t-s|^{1/2+\varepsilon}
\notag
\\&
\lesssim [f]_{\alpha}
n^{-\frac{1+\alpha}{2} + \varepsilon}|t - s|^{1/2 + \varepsilon}.
\end{align}
Next, we consider the integral on $(t_{k+2},t]$.
Let $r \in (t_{k + 2}, t]$ with $r-\kappa_{n}(r)>0$.
Then since $\kappa_n(r) > t_{k + 1} > s$, we have
\begin{align*}
&\bE^{s}\Big[
\Big\{f(Y_r^n) - f(Y_{\kappa_n(r)}^n)\Big\}
1_{\{|Y_r^n - Y_{\kappa_n(r)}^n| \leq 1\}}
\Big]
\\&=
\bE^{s}\Big[
\bE^{t_{k+1}}\Big[
\bE^{\kappa_{n}(r)}\Big[
\Big\{f(Y_r^n) - f(Y_{\kappa_n(r)}^n)\Big\}
1_{\{|Y_r^n - Y_{\kappa_n(r)}^n| \leq 1\}}
\Big]
\Big]
\Big]
\\&=
\bE^{s}[h(Y_{t_{k + 1}}^n)],
\end{align*}
where the function $h:\bR^{d} \to \bR$ is defined by
\begin{align*}
h(w)
:=
\int_{\bR^d}
\rmd z
\int_{|y| \leq 1}
\rmd y\,
\Big\{f(z + y) - f(z)\Big\}\,
g_{a(z)(r-\kappa_{n}(r))}(y)\,
p^n(t_{k + 1}, w, \kappa_n(r), z),~w \in \bR^{d},
\end{align*}
and $p^n(t_{k + 1}, w, \kappa_n(r), \cdot)$ is a transition density of $Y^{n}$.
Note that the diffusion coefficient $\sigma$ is assumed to be time independent, thus $(Y^{n}_{t_{k}})_{k=0,\ldots,n}$ is a time inhomogeneous Markov chain, which implies that  $p^n(t_{k + 1}, w, \kappa_n(r), \cdot)$ equals to the density function $p^{n}_{\kappa_{n}(r)-t_{k+1}}(w,\cdot)$ of $Y^{n}_{\kappa_{n}(r)-t_{k+1}}$ with $Y^{n}_{0}=w$.
By the assumptions on the diffusion coefficient $\sigma$ and \cite[Theorem 2.1]{LeMe10}, there exist constants $C, c > 0$ such that for any $0 \leq i < j \leq n$,
\begin{equation*}
g_{a(z)(r-\kappa_{n}(r))}(y)
\leq
C g_{c(r-\kappa_{n}(r))}(y)
\text{ and }
p^n(t_{i}, w, t_{j}, z)
\leq
Cg_{c(t_j - t_i)}(z-w).
\end{equation*}
Therefore, by using the mean value theorem and the integration by parts formula, we have
\begin{align}\label{h_w}
h(w)
&=
\int_{\bR^d} \rmd z
\int_{|y| \leq 1} \rmd y
\int_{0}^{1} \rmd \theta\,
\big\langle y,(\nabla_{z}f)(z + \theta y) \big\rangle\,
g_{a(z)(r-\kappa_{n}(r))}(y)\,
p^n(t_{k + 1}, w, \kappa_n(r), z)
\notag
\\&=
-
\int_{\bR^d} \rmd z
\int_{|y| \leq 1} \rmd y
\int_{0}^{1} \rmd \theta\,
f(z + \theta y)\,
\big\langle
y,
\nabla_{z}
\{g_{a(z)(r-\kappa_{n}(r))}(y)\,p^n(t_{k + 1}, w, \kappa_n(r), z)\}
\big\rangle \notag
\\&=
-\{h_{1}(w)+h_{2}(w)\},
\end{align}
where $h_{1},h_{2}:\bR^{d} \to \bR$ are defined by
\begin{align*}
h_{1}(w)
&:=
\int_{\bR^d} \rmd z
\int_{|y| \leq 1} \rmd y
\int_{0}^{1} \rmd \theta\,
f(z + \theta y)\,
\big\langle
y,
\nabla_{z}g_{a(z)(r-\kappa_{n}(r))}(y)
\big\rangle\,
p^n(t_{k + 1}, w, \kappa_n(r), z),
\\
h_{2}(w)
&:=
\int_{\bR^d} \rmd z
\int_{|y| \leq 1} \rmd y
\int_{0}^{1} \rmd \theta\,
f(z + \theta y)\,
g_{a(z)(r-\kappa_{n}(r))}(y)\,
\big\langle
y,
(\nabla_{z} p^n)(t_{k + 1}, w, \kappa_n(r), z)
\big\rangle.
\end{align*}
We first consider $h_{1}$.
For each $j=1,\ldots,d$, it holds that
\begin{align*}
\partial_{z_j}g_{a(z)(r-\kappa_{n}(r))}(y)
&=
\frac{1}{(2\pi (r-\kappa_{n}(r)))^{d/2}}
\partial_{z_j}\Big\{
\{\det a(z)\}^{-1/2} \exp\Big(-\frac{\langle a(z)^{-1}y, y \rangle}{2(r-\kappa_{n}(r))}\Big)
\Big\}
\notag
\\&=
-\frac{\det(a(z))^{-3/2}}{2(2\pi (r-\kappa_{n}(r)))^{d/2}}
(\partial_{z_j}\det a(z))\,
\exp\Big(-\frac{\langle a(z)^{-1}y, y \rangle}{2(r-\kappa_{n}(r))}\Big)
\notag
\\&\quad-
\frac{\det(a(z))^{-1/2}}{(2\pi (r-\kappa_{n}(r)))^{d/2}}
\frac{\partial_{z_j}\langle a(z)^{-1}y, y \rangle}{2(r-\kappa_{n}(r))}
\exp\Big(-\frac{\langle a(z)^{-1}y, y \rangle}{2(r-\kappa_{n}(r))}\Big)
\notag
\\&=
-\frac{1}{2}
\Big\{
\frac{\partial_{z_j}\det a(z)}{\det a(z)}
+
\frac{\partial_{z_j}\langle a(z)^{-1}y, y \rangle}{r-\kappa_{n}(r)}
\Big\}
g_{a(z)(r-\kappa_{n}(r))}(y).
\end{align*}
Hence we have
\begin{align*}
h_{1}(w)
=
-\frac{1}{2}
\Big\{h_{1,1}(w)+h_{1,2}(w)\Big\},
\end{align*}
where $h_{1,1},h_{1,2}:\bR^{d} \to \bR$ are defined by
\begin{align*}
h_{1,1}(w)
&:=
\int_{\bR^d} \rmd z
\int_{|y| \leq 1} \rmd y
\int_{0}^{1} \rmd \theta\,
f(z + \theta y)\,
\frac{\langle y,\nabla_{z}\det a(z)\rangle}{\det a(z)}\,
g_{a(z)(r-\kappa_{n}(r))}(y)\,
p^n(t_{k + 1}, w, \kappa_n(r), z),
\\
h_{1,2}(w)
&:=
\int_{\bR^d} \rmd z
\int_{|y| \leq 1} \rmd y
\int_{0}^{1} \rmd \theta\,
f(z + \theta y)\,
\frac{\langle y,\nabla_{z}\langle a(z)^{-1}y, y \rangle \rangle}{r-\kappa_{n}(r)}\,
g_{a(z)(r-\kappa_{n}(r))}(y)\,
p^n(t_{k + 1}, w, \kappa_n(r), z).
\end{align*}
We estimate $h_{1,1}$.
Note that for each $j=1,\ldots,d$,
\begin{align*}
\int_{\bR^d} \rmd z
\int_{|y| \leq 1} \rmd y
\int_{0}^{1} \rmd \theta\,
f(z)\,
\frac{y_{j}\,\partial_{z_j}\det a(z)}{\det a(z)}\,
g_{a(z)(r-\kappa_{n}(r))}(y)\,
p^n(t_{k + 1}, w, \kappa_n(r), z)
=0
\end{align*}
and the assumptions on the diffusion coefficient $\sigma$, the map $\bR^{d} \ni z \mapsto \nabla_{z}\det a(z)/\det a(z) \in \bR^{d}$ is bounded.
Hence, by using the uniformly local H\"older continuity of $f$, we have
\begin{align}\label{H_1_1}
|h_{1,1}(w)|
&=
\Big|
\int_{\bR^d} \rmd z
\int_{|y| \leq 1} \rmd y
\int_{0}^{1} \rmd \theta\,
\{f(z + \theta y)-f(z)\}
\notag
\\&\hspace{3cm}
\times
\frac{\langle y, \nabla_{z}\det a(z)\rangle}{\det a(z)}\,
g_{a(z)(r-\kappa_{n}(r))}(y)\,
p^n(t_{k + 1}, w, \kappa_n(r), z)
\Big|
\notag
\\&\lesssim
\int_{\bR^d} \rmd z
\int_{|y| \leq 1} \rmd y
\int_{0}^{1} \rmd \theta\,
|f(z + \theta y)-f(z)|
|y|\,
g_{c(r-\kappa_{n}(r))}(y)\,
p^n(t_{k + 1}, w, \kappa_n(r), z)
\notag
\\&\leq
[f]_{\alpha}
\int_{\bR^d} \rmd z
\int_{|y| \leq 1} \rmd y\,
|y|^{1+\alpha}\,
g_{c(r-\kappa_{n}(r))}(y)\,
p^n(t_{k + 1}, w, \kappa_n(r), z)
\notag
\\&\lesssim
[f]_{\alpha}n^{-\frac{1+\alpha}{2}}.
\end{align}
Next we estimate $h_{1,2}$.
Note that for each $j=1,\ldots,d$,
\begin{align*}
&\int_{\bR^d} \rmd z
\int_{|y| \leq 1} \rmd y
\int_{0}^{1} \rmd \theta\,
f(z)\,
\frac{y_{j}\,\partial_{z_j}\langle a(z)^{-1}y, y \rangle}{r-\kappa_{n}(r)}\,
g_{a(z)(r-\kappa_{n}(r))}(y)\,
p^n(t_{k + 1}, w, \kappa_n(r), z)
\\&=
\sum_{m,\ell=1}^{d}
\int_{\bR^d} \rmd z
\int_{|y| \leq 1} \rmd y
\int_{0}^{1} \rmd \theta\,
f(z)\,
y_{j}y_{m}y_{\ell}
\frac{\partial_{z_j}(a(z)^{-1})_{m\ell}}{r-\kappa_{n}(r)}\,
g_{a(z)(r-\kappa_{n}(r))}(y)\,
p^n(t_{k + 1}, w, \kappa_n(r), z)
\\&=0
\end{align*}
and by the assumptions on the diffusion coefficient $\sigma$, we have
\begin{align*}
|\nabla_{z}\langle a(z)^{-1}y, y \rangle|
\lesssim
|y|^{2}.
\end{align*}
Hence, by using the uniformly local H\"older continuity of $f$, we have
\begin{align}\label{H_1_2}
|h_{1,2}(w)|
&=
\Big|\int_{\bR^d} \rmd z
\int_{|y| \leq 1} \rmd y
\int_{0}^{1} \rmd \theta\,
\{f(z + \theta y)-f(z)\}
\notag
\\&\hspace{3cm}
\times
\frac{\langle y,\nabla_{z}\langle a(z)^{-1}y, y \rangle \rangle}{r-\kappa_{n}(r)}\,
g_{a(z)(r-\kappa_{n}(r))}(y)\,
p^n(t_{k + 1}, w, \kappa_n(r), z)
\Big|
\notag
\\&\lesssim
\frac{[f]_{\alpha}}{r-\kappa_{n}(r)}
\int_{\bR^d} \rmd z
\int_{|y| \leq 1} \rmd y\,
|y|^{3+\alpha}
g_{c(r-\kappa_{n}(r))}(y)\,
p^n(t_{k + 1}, w, \kappa_n(r), z)
\notag
\\&\lesssim
[f]_{\alpha}
n^{-\frac{1+\alpha}{2}}.
\end{align}
Therefore, by \eqref{H_1_1} and \eqref{H_1_2}, we have
\begin{equation}\label{est_H_1}
\sup_{w \in \bR^{d}}|h_1(w)|
\lesssim
[f]_{\alpha}
n^{-\frac{1 + \alpha}{2}}.
\end{equation}
Next, we estimate $h_{2}$.
Note that
\begin{align*}
\int_{\bR^d} \rmd z
\int_{|y| \leq 1} \rmd y
\int_{0}^{1} \rmd \theta\,
f(z)\,
g_{a(z)(r-\kappa_{n}(r))}(y)\,
\big\langle
y,
(\nabla_{z} p^n)(t_{k + 1}, w, \kappa_n(r), z)
\big\rangle
=
0
\end{align*}
and by \cite[page 278, line 9]{KoMa02}, there exists constants $C,c>0$ such that
\begin{equation*}
|(\nabla_{z} p^n)(t_{k + 1}, w, \kappa_n(r), z)|
\leq
\frac{C}{(\kappa_n(r) - t_{k + 1})^{1/2}}
g_{c(\kappa_n(r) - t_{k + 1})}(z - w).
\end{equation*}
Hence by using the uniformly local H\"older continuity of $f$, we have
\begin{align}\label{est_H_2}
|h_{2}(w)|
&=
\Big|
\int_{\bR^d} \rmd z
\int_{|y| \leq 1} \rmd y
\int_{0}^{1} \rmd \theta\,
\{f(z + \theta y)-f(z)\}\,
g_{a(z)(r-\kappa_{n}(r))}(y)\,
\big\langle
y,
(\nabla_{z} p^n)(t_{k + 1}, w, \kappa_n(r), z)
\big\rangle
\Big|
\notag
\\&\lesssim
\frac{[f]_{\alpha}}{(\kappa_n(r) - t_{k + 1})^{1/2}}
\int_{\bR^d} \rmd z
\int_{|y| \leq 1} \rmd y
\int_{0}^{1} \rmd \theta\,
|y|^{1+\alpha}\,
g_{a(z)(r-\kappa_{n}(r))}(y)\,
g_{c(\kappa_n(r) - t_{k + 1})}(z - w).
\notag
\\&\lesssim
\frac{[f]_{\alpha}n^{-\frac{1+\alpha}{2}}}{(\kappa_n(r) - t_{k + 1})^{1/2}}.
\end{align}
Therefore, combining \eqref{h_w}, \eqref{est_H_1} and \eqref{est_H_2}, we obtain
\begin{align*}
\sup_{w \in \bR^d}|h(w)|
\lesssim
[f]_{\alpha}
\Big\{
1 + \frac{1}{(\kappa_n(r) - t_{k + 1})^{1/2}}
\Big\}
n^{-\frac{1 + \alpha}{2}}.
\end{align*}
This implies that
\begin{align}\label{J_2}
&\int_{t_{k+2}}^{t}
\Big\|
\bE^{s}\Big[
\Big\{f(Y_r^n) - f(Y_{\kappa_n(r)}^n)\Big\}
1_{\{|Y_r^n - Y_{\kappa_n(r)}^n| \leq 1\}}
\Big]
\Big\|_{L^{p}}
\rd r
\notag
\\&\lesssim
[f]_{\alpha}
n^{-\frac{1 + \alpha}{2}}
\int_{t_{k+2}}^{t}
\Big\{
1+\frac{1}{(\kappa_n(r) - t_{k + 1})^{1/2}}
\Big\}
\rd r
\notag
\\&\leq
[f]_{\alpha}
n^{-\frac{1 + \alpha}{2}}
\int_{t_{k+2}}^{t}
\Big\{
1+\frac{1}{(r - t_{k + 2})^{1/2}}
\Big\}
\rd r
=
[f]_{\alpha}
n^{-\frac{1 + \alpha}{2}}
\Big\{
(t-t_{k+2})
+
2(t- t_{k +2})^{1/2}
\Big\}
\notag
\\&\lesssim
[f]_{\alpha}
n^{-\frac{1 + \alpha}{2}}
(t-t_{k+2})^{1/2}
\lesssim
[f]_{\alpha}
n^{-\frac{1 + \alpha}{2}+\varepsilon}
|t-s|^{1/2+\varepsilon}.
\end{align}
Therefore, by \eqref{A_st_1}, \eqref{J_1} and \eqref{J_2}, we obtain for $t > t_{k + 3}$,
\begin{align*}
\|A_{s,t}\|_{L^{p}}
&\lesssim
([f]_{\alpha} \vee |f(0)|)
\|g\|_\infty
\Big\{
n^{-\frac{1 + \alpha}{2}+\varepsilon}
+
\exp\Big(-\frac{n}{8 p \lambda T}\Big)
\Big\}
|t-s|^{1/2+\varepsilon}
\notag
\\&\lesssim
([f]_{\alpha} \vee |f(0)|)
\|g\|_\infty
n^{-\frac{1 + \alpha}{2}+\varepsilon}
|t-s|^{1/2+\varepsilon}.
\end{align*}
For $t \leq t_{k + 3}$, by a similar argument as \eqref{J_1} yields the same bound as the above.
Therefore, the condition \eqref{sew_1} holds with $\varepsilon_1 = \varepsilon$ and $\Gamma_{1} =C([f]_{\alpha} \vee |f(0)|)\|g\|_\infty  n^{-\frac{1 + \alpha}{2} + \varepsilon}$ for some constant $C> 0$.

Next we prove that $(A_{s, t})_{0 \leq s \leq  t \leq T}$ satisfies the condition \eqref{sew_2}.
By the definition of $(A_{s, t})_{0 \leq s \leq  t \leq T}$, we have for any $0 \leq s \leq u \leq t \leq T$,
\begin{align*}
\bE^s[A_{s,t}-A_{s,u}-A_{u,t}]
&=
\int_u^t
\bE^{s}\Big[
\Big\{g(Y_s^n)-g(Y_u^n)\Big\}
\bE^{u}
\Big[
f(Y_r^n)-f(Y_{\kappa_n(r)}^n)
\Big]
\Big]
\rd r.
\end{align*}
Hence by using H\"older's inequality, Jensen's inequality, Lemma \ref{lem_dif_em}, Lemma \ref{lem_tail}, \eqref{J_1} and \eqref{J_2}, we have
\begin{align*}
\|\bE^s[A_{s,t}-A_{s,u}-A_{u,t}]\|_{L^{p}}
&\leq
\int_u^t
\Big\|
\bE^{s}\Big[
\Big\{g(Y_s^n)-g(Y_u^n)\Big\}
\bE^{u}
\Big[
f(Y_r^n)-f(Y_{\kappa_n(r)}^n)
\Big]
\Big]
\Big\|_{L^{p}}
\rd r
\notag
\\&\leq
\|
g(Y_s^n)-g(Y_u^n)
\|_{L^{2p}}
\int_u^t
\Big\|
\bE^{u}
\Big[
f(Y_r^n)-f(Y_{\kappa_n(r)}^n)
\Big]
\Big\|_{L^{2p}}
\rd r
\notag
\\&\leq
[g]_{\mathrm{Lip}}
\|
Y_s^n-Y_u^n
\|_{L^{2p}}
\int_u^t
\Big\|
\bE^{u}
\Big[
f(Y_r^n)-f(Y_{\kappa_n(r)}^n)
\Big]
\Big\|_{L^{2p}}
\rd r
\notag
\\&\lesssim
([f]_{\alpha} \vee|f(0)|)
[g]_{\mathrm{Lip}}
n^{-\frac{1 + \alpha}{2} + \varepsilon}|t - s|^{1 + \varepsilon}.
\end{align*}
Therefore, the condition \eqref{sew_2} holds with $\varepsilon_{2} = \varepsilon$ and $\Gamma_2 = C([f]_{\alpha} \vee|f(0)|)[g]_{\mathrm{Lip}} n^{-\frac{1 + \alpha}{2} + \varepsilon}$ for some constant $C > 0$.

Therefore, from the stochastic sewing lemma (Lemma \ref{lem_stochastic_sewing}), there exist a unique stochastic process $(\cA_t)_{t \in [0,T]}$ satisfying the conditions (i), (ii) in Lemma \ref{lem_stochastic_sewing}, and in particular \eqref{sew_5} holds.
We define a stochastic process $(\widetilde{\cA}_{t})_{t \in [0,T]}$ by
\begin{align*}
\widetilde{\cA}_{t}
:=
\int_0^t
\Big\{f(Y_r^n) - f(Y_{\kappa_n(r)}^n)\Big\}
g(Y_r^n)
\rd r,~t \in [0,T].
\end{align*}
Then we prove $\cA_{t}=\widetilde{\cA}_{t}$. a.s. for $t \in [0,T]$.
By the definition, it is trivial that $\widetilde{\cA}$ satisfies (i) in Lemma \ref{lem_stochastic_sewing}.
Next we prove that $\widetilde{\cA}$ satisfies the conditions \eqref{sew_3} and \eqref{sew_4} with some constants $\Gamma_{3},\Gamma_{4}\geq 0$.
By Lemma \ref{lem_p_em}, for any $0 \leq s \leq t \leq T$, we have
\begin{align*}
&\|\widetilde{\cA}_t - \widetilde{\cA}_s - A_{s, t}\|_{L^p}
\leq
\|\widetilde{\cA}_t - \widetilde{\cA}_s\|_{L^p}
+
\|A_{s, t}\|_{L^p}
\\&\leq
\Big\|\int_s^t \Big\{f(Y_r^n) - f(Y_{\kappa_n(r)}^n)\Big\}g(Y_r^n) \rd r\Big\|_{L^p}
+
\Big\|\bE\Big[\int_s^t \Big\{f(Y_r^n) - f(Y_{\kappa_n(r)}^n)\Big\}g(Y_{s}^n) \rd r\Big]\Big\|_{L^p}
\\&\lesssim
|t - s|
=
|t - s|^{1/2 - \varepsilon}|t - s|^{1/2 + \varepsilon}
\leq
T^{1/2 - \varepsilon}|t - s|^{1/2 + \varepsilon}.
\end{align*}
Hence, the condition \eqref{sew_3} holds with some $\Gamma_{3}>0$ and any $\Gamma_4\geq 0$.
Since $f$ is of linear growth, by using H\"older's inequality and Lemma \ref{lem_dif_em}, we have
\begin{align*}
\|\bE^{s}[\widetilde{\cA}_t - \widetilde{\cA}_s - A_{s, t}]\|_{L^p}
&\leq
\int_{s}^{t}
\Big\|
\Big\{
g(Y_r^n) - g(Y_{s}^n)
\Big\}
\Big\{
f(Y_r^n) - f(Y_{\kappa_n(r)}^n)
\Big\}
\Big\|_{L^{p}}
\rd r
\\&\leq
\int_{s}^{t}
\Big\|
g(Y_r^n) - g(Y_{s}^n)
\Big\|_{L^{2p}}
\Big\|
f(Y_r^n) - f(Y_{\kappa_n(r)}^n)
\Big\|_{L^{2p}}
\rd r
\\&\lesssim
\int_{s}^{t}
\|Y_r^n-Y_{s}^n\|_{L^{2p}}
\rd r
\lesssim
\int_{s}^{t}
(r-s)^{1/2}
\rd r
\lesssim
(t-s)^{1+\varepsilon}.
\end{align*}
Hence, the condition \eqref{sew_4} holds with some constant $\Gamma_{4}>0$.
Therefore, by the uniqueness of the stochastic process $\cA$, we have $\cA_{t} = \widetilde{\cA}_{t}$, a.s., $t \in [0,T]$.
In particular, since $\cA$ satisfies the condition \eqref{sew_5}, we obtain
\begin{align*}
\sup_{t \in [0,T]}
\Big\|
\int_{0}^{t}
\Big\{f(Y_{r}^{n})-f(Y_{\kappa_n(r)}^{n})\Big\}
g(Y_{r}^{n})
\rd r
\Big\|_{L^{p}}
&\lesssim
([f]_{\alpha} \vee |f(0)|)
(\|g\|_\infty+[g]_{\mathrm{Lip}})
n^{-\frac{1 + \alpha}{2} + \varepsilon}.
\end{align*}
This concludes the proof.
\end{proof}

Finally, we extend the above proposition to the Euler--Maruyama scheme $X^{n}$.

\begin{Prop}\label{lem_quadratic_em}
Let $f \in C^{\alpha}(\bR^{d};\bR)$ with $\alpha \in (0,1]$ and $g:\bR^{d} \to \bR$ be a bounded and global Lipschitz continuous function.
We suppose that Assumption \ref{asm_01} holds.
Then for any $p \geq 1$ and $\varepsilon \in (0,(1+\alpha))/2)$, there exists a constant $C_{p,\varepsilon} > 0$ such that for any $t \in [0, T]$ and $n \in \bN$,
\begin{align*}
\sup_{t \in [0,T]}
\Big\|
\int_{0}^{t}
\Big\{f(X_{r}^{n})-f(X_{\kappa_n(r)}^{n})\Big\}
g(X_{r}^{n})
\rd r
\Big\|_{L^{p}}
\leq
C_{p,\varepsilon}
([f]_{\alpha} \vee |f(0)|)(\|g\|_\infty + [g]_{\mathrm{Lip}}) n^{-\frac{1 + \alpha}{2} + \varepsilon}.
\end{align*}
\end{Prop}

In order to prove this proposition, we use Maruyama--Girsanov transformation.
For $q \in \bR$ and $n \in \bN$, we define a stochastic process $Z(q,n)=(Z_{t}(q,n))_{t \in [0,T]}$ by
\begin{align*}
Z_{t}(q,n)
&:=
\exp\Big(
q
\int_{0}^{t}
\langle
\sigma(Y^{n}_{\kappa_{n}(s)})^{-1}
b(Y^{n}_{\kappa_{n}(s)}),
\rd B_{s}
\rangle
-
\frac{q^{2}}{2}
\int_{0}^{t}
|\sigma(Y^{n}_{\kappa_{n}(s)})^{-1}
b(Y^{n}_{\kappa_{n}(s)})|^2
\rd s
\Big).
\end{align*}
Then we need the following moment estimate for $Z(q,n)$.

\begin{Lem}\label{lem_Girsanov}
We suppose that Assumption \ref{asm_01} (ii) and (iv) hold.
Then for any $p>0$,
\begin{align*}
\sup_{n \in \bN}
\bE\Big[
\exp\Big(
p
\int_{0}^{T}
|\sigma(Y^{n}_{\kappa_{n}(s)})^{-1}
b(Y^{n}_{\kappa_{n}(s)})|^2
\rd s
\Big)
\Big]
<\infty.
\end{align*}
In particular, for any $q \in \bR$ and $n \in \bN$, the stochastic process $Z(q,n)=(Z_{t}(q,n))_{t \in [0,T]}$ is a martingale and for any $p>0$,
\begin{align*}
\sup_{n \in \bN}
\sup_{t \in [0,T]}
\bE[Z_{t}(q,n)^{p}]
<\infty.
\end{align*}
\end{Lem}
\begin{proof}
We set $M_{t}^{n}:=\int_0^{t} \sigma(Y_{\kappa_n(r)}^n)\rd B_r$, $t \in[0,T]$.
By the definition of the Euler--Maruyama scheme $Y^n$, we have for any $k=0,\ldots,n$,
\begin{align*}
|Y_{t_k}^n|^2
&\leq
2|x_0|^2
+
2\sup_{t \in [0,T]}|M_{t}^{n}|^2.
\end{align*}
Since $b$ is of sub-linear growth, for any $\delta > 0$, there exists $L_\delta > 0$ such that for any $x \in \bR^d$
\begin{equation*}
|b(x)| \le \delta|x| + L_\delta.
\end{equation*}
Hence we have
\begin{align*}
\int_0^T
|\sigma(Y_{\kappa_n(r)}^n)^{-1}
b(Y_{\kappa_n(r)}^n)|^2
\rd r
&\lesssim
\int_0^T
\Big\{
\delta^{2}|Y_{\kappa_n(r)}^n|^{2}+L_\delta^{2}
\Big\}
\rd r
\lesssim
(L_{\delta}^{2}+\delta^{2})
+
\delta^{2}
\sup_{t \in [0,T]}|M_{t}^{n}|^2.
\end{align*}
Now we fix $p \geq 2$.
Then there exists $C_{p}>0$ such that,
\begin{align*}
\bE\Big[
\exp\Big(
p
\int_{0}^{T}
|\sigma(Y^{n}_{\kappa_{n}(s)})^{-1}
b(Y^{n}_{\kappa_{n}(s)})|^2
\rd s
\Big)
\Big]
&\leq
\exp(C_{p}(L_{\delta}^{2}+\delta^{2}))
\bE\Big[
\exp\Big(
C_{p}\delta^{2}
\sup_{t \in [0,T]}|M_{t}^{n}|^{2}
\Big)
\Big].
\end{align*}
By using Burkholder--Davis--Gundy's inequality (see \cite{BaYo82, Re08}) with a sharp constant
\begin{align*}
A_{q}:=K_{d} q^{q/2},~q \geq 2,
\end{align*}
for some constant $K_{d}>0$ depending only on the dimension $d$, we have for any $q \geq 2$
\begin{align*}
\bE\Big[
\sup_{t \in [0,T]}|M_{t}^{n}|^{q}
\Big]
\leq
A_{q}
\bE\Big[\Big(
\int_0^T
|\sigma(Y_{\kappa_n(r)}^n)|^2 \rd r
\Big)^{q/2}\Big]
\leq
A_{q}
\|\sigma\|_{\infty}^{q} T^{q/2}.
\end{align*}
Then by using Taylor expansion for the exponential function and the monotone convergence theorem, we have
\begin{align*}
\bE\Big[
\exp\Big(
C_{p}\delta^{2}
\sup_{t \in [0,T]}|M_{t}^{n}|^{2}
\Big)
\Big]
&=
\sum_{k=0}^{\infty}
\frac{C_{p}^{k}\delta^{2k}}{k!}
\bE\Big[
\sup_{t \in [0,T]}|M_{t}^{n}|^{2k}
\Big]
\leq
1
+
K_{d}
\sum_{k=1}^{\infty}
\frac{(2C_{p}\|\sigma\|_{\infty}^{2} Tk\delta^{2})^{k}}{k!}.
\end{align*}
By Stirling's formula $\sqrt{2\pi}k^{k+1/2}e^{-k}\leq k!$ for any $k \in \bN$, we have
\begin{align*}
\sum_{k=1}^{\infty}
\frac{(2C_{p}\|\sigma\|_{\infty}^{2} Tk)^{k}\delta^{2k}}{k!}
&\leq
\frac{1}{\sqrt{2\pi}}
\sum_{k=1}^{\infty}
\frac{(2C_{p}\|\sigma\|_{\infty}^{2} Te\delta^{2})^{k}}{k^{1/2}}.
\end{align*}
Therefore, by choosing $\delta$ as
$
\delta
=
2^{-1}(C_{p}\|\sigma\|_{\infty}^{2} Te)^{-1/2},
$
we have
\begin{align}\label{eq_sum_0}
\sum_{k=1}^{\infty}
\frac{(2C_{p}\|\sigma\|_{\infty}^{2} Tk\delta^{2})^{k}}{k!}
<\infty.
\end{align}
This implies that
\begin{align*}
\sup_{n \in \bN}
\bE\Big[
\exp\Big(
p
\int_{0}^{T}
|\sigma(Y^{n}_{\kappa_{n}(s)})^{-1}
b(Y^{n}_{\kappa_{n}(s)})|^2
\rd s
\Big)
\Big]
<\infty.
\end{align*}
In particular, by Novikov condition, for $q \in \bR$ and $n \in \bN$, $Z(q,n)=(Z_{t}(q,n))_{t \in [0,T]}$ is a martingale.

By using H\"older's inequality, we have
\begin{align*}
\bE[Z_{t}(q,n)^{p}]
&=
\bE\Big[
\exp\Big(
pq
\int_{0}^{t}
\langle
\sigma(Y^{n}_{\kappa_{n}(s)})^{-1}
b(Y^{n}_{\kappa_{n}(s)}),
\rd B_{s}
\rangle
-
\frac{pq^{2}}{2}
\int_{0}^{t}
|\sigma(Y^{n}_{\kappa_{n}(s)})^{-1}
b(Y^{n}_{\kappa_{n}(s)})|^2
\rd s
\Big)
\Big]
\\&=
\bE\Big[
Z_{t}(2pq,n)^{1/2}
\exp\Big(
\frac{(2p^{2}-p)q^{2}}{2}
\int_{0}^{t}
|\sigma(Y^{n}_{\kappa_{n}(s)})^{-1}
b(Y^{n}_{\kappa_{n}(s)})|^2
\rd s
\Big)
\Big]
\\&\leq
\bE\Big[
Z_{t}(2pq,n)
\Big]^{1/2}
\bE\Big[
\exp\Big(
(2p^{2}-p)q^{2}
\int_{0}^{t}
|\sigma(Y^{n}_{\kappa_{n}(s)})^{-1}
b(Y^{n}_{\kappa_{n}(s)})|^2
\rd s
\Big)
\Big]^{1/2}.
\end{align*}
Since $Z(2pq,n)$ is a martingale, we have $\bE[Z_{t}(2pq,n)]=1$.
Hence we obtain
\begin{align*}
\sup_{n \in \bN}
\sup_{t \in [0,T]}\bE[Z_{t}(q,n)^{p}]<\infty.
\end{align*}
This concludes the proof.
\end{proof}

\begin{Rem}\label{Rem_linear}
If we replace Assumption \ref{asm_01} (ii) by the linear growth condition: there exists $L>0$, such that $|b(x)|\leq L(1+|x|)$, $x \in \bR^{d}$, then we can prove the statement in Theorem \ref{main_thm00} for sufficiently small $T>0$.
Indeed, if $T$ is sufficiently small, then the series \eqref{eq_sum_0} convergences with $\delta=L$.
Babi, Dieye and Menoukeu Pamen \cite{BaDiPa23} use this argument in their proof.

\end{Rem}

Now we are ready to prove Proposition \ref{lem_quadratic_em}.

\begin{proof}[Proof of Proposition \ref{lem_quadratic_em}]
By H\"older's inequality, it suffices to prove the bound for $p \geq 2$.
Let $t \in [0,T]$ be fixed.
Then by Lemma \ref{lem_Girsanov}, Maruyama--Girsanov theorem and H\"older's inequality, we have
\begin{align*}
&\bE\Big[
\Big|
\int_{0}^{t}
\Big\{f(X_{r}^{n})-f(X_{\kappa_n(r)}^{n})\Big\}
g(X_{r}^{n})
\rd r
\Big|^{p}
\Big]
=
\bE\Big[
\Big|
\int_{0}^{t}
\Big\{f(Y_{r}^{n})-f(Y_{\kappa_n(r)}^{n})\Big\}
g(Y_{r}^{n})
\rd r
\Big|^{p}
Z_{t}(1,n)
\Big]
\\&\leq
\bE\Big[
\Big|
\int_{0}^{t}
\Big\{f(Y_{r}^{n})-f(Y_{\kappa_n(r)}^{n})\Big\}
g(Y_{r}^{n})
\rd r
\Big|^{2p}
\Big]^{1/2}
\bE[Z_{t}(1,n)^{2}]^{1/2}.
\end{align*}
Therefore, from Proposition \ref{lem_Lp} and Lemma \ref{lem_Girsanov}, we have
\begin{align*}
\sup_{t \in [0,T]}
\Big\|
\int_{0}^{t}
\Big\{f(X_{r}^{n})-f(X_{\kappa_n(r)}^{n})\Big\}
g(X_{r}^{n})
\rd r
\Big\|_{L^{p}}
\lesssim
([f]_{\alpha} \vee |f(0)|)(\|g\|_\infty + [g]_{\mathrm{Lip}}) n^{-\frac{1 + \alpha}{2} + \varepsilon}.
\end{align*}
This concludes the proof.
\end{proof}

\subsection{Proof of main results}\label{Sec_3_3}

Let $f: \bR^d \to \bR^d$ be a measurable function and $\lambda > 0$.
We consider the following elliptic equation
\begin{equation}\label{pde_01}
\lambda u -\sL u = f,
\end{equation}
where the operator $\sL$ is the infinitesimal generator of $X$, that is,
\begin{align*}
(\sL g)(x)
:=
\frac{1}{2}\mathrm{Tr}\big(a(x) (\nabla^{2}g)(x)\big)
+
\langle b(x),(\nabla g)(x) \rangle,~a(a)=\sigma(x)\sigma(x)^{\top}.
\end{align*}
The following lemma shows that if $f \in C^{\alpha}(\bR^{d};\bR^{d})$, $\alpha \in (0,1)$, then the elliptic equation \eqref{pde_01} admits a classical solution.

\begin{Lem}[Theorem 5 in \cite{FlGuPr10}]\label{lem_pde_01}
We suppose that Assumption \ref{asm_01} (i), (iii) and (iv) hold.
Then for any $\alpha^\prime \in (0, \alpha)$, there exists $\lambda_0 > 0$ such that for any $\lambda \geq \lambda_0$, for any $f \in C^{\alpha}(\bR^d;\bR^d)$, the equation \eqref{pde_01} admits a unique classical solution $u \in C^{2 + \alpha^\prime}(\bR^d; \bR^d)$ such that
\begin{equation*}
\|u\|_{2 + \alpha^\prime}
=
\|u(\cdot)(1 + |\cdot|)^{-1}\|_\infty
+
\|\nabla u\|_\infty
+
\|\nabla^2 u\|_\infty
+
[\nabla^2 u]_{\alpha^\prime}
\leq
C(\lambda)
\|f\|_{\alpha},
\end{equation*}
where the constant $C(\lambda)$ is independent of $u$ and $f$, and $C(\lambda) \to 0$ as $\lambda \to \infty$.
\end{Lem}

Now we are ready to prove Theorem \ref{main_thm00}.

\begin{proof}[Proof of Theorem \ref{main_thm00}]
It suffices to prove the bound for $p \ge 2$ and $\varepsilon \in (0, \alpha/2)$.
Let $t \in [0, T]$ be fixed.
Then it holds that
\begin{align*}
\bE[|X_t - X_t^n|^p]
&\lesssim
\Big\{
I_{1}^{n}(t)
+
I_{2}^{n}(t)
+
I_{3}^{n}(t)
+
I_{4}^{n}(t)
\Big\},
\end{align*}
where $I_{1}^{n}(t)$, $I_{2}^{n}(t)$, $I_{3}^{n}(t)$ and $I_{4}^{n}(t)$ are defined by
\begin{align*}
I_{1}^{n}(t)
&:=
\bE\Big[
\Big|
\int_{0}^{t}
\Big\{
b(X_{r})
-
b(X_{r}^{n})
\Big\}
\rd r
\Big|^{p}
\Big],~
I_{2}^{n}(t)
:=
\bE\Big[
\Big|
\int_{0}^{t}
\Big\{
b(X_{r}^{n})
-
b(X_{\kappa_{n}(r)}^{n})
\Big\}
\rd r
\Big|^{p}
\Big],
\\
I_{3}^{n}(t)
&:=
\bE\Big[
\Big|
\int_{0}^{t}
\Big\{
\sigma(X_{r})
-
\sigma(X_{r}^{n})
\Big\}
\rd B_{r}
\Big|^{p}
\Big],~
I_{4}^{n}(t)
:=
\bE\Big[
\Big|
\int_{0}^{t}
\Big\{
\sigma(X_{r}^{n})
-
\sigma(X_{\kappa_{n}(r)}^{n})
\Big\}
\rd B_{r}
\Big|^{p}
\Big].
\end{align*}
For $\alpha=1$ and $\sigma$ is a constant matrix, then since $b$ is globally Lipschitz continuous, by using Proposition \ref{lem_quadratic_em} with $g=1$, we have
\begin{align*}
\bE[|X_t - X_t^n|^p]
&\lesssim
\int_{0}^{t}
\bE[|X_s - X_s^n|^p]
\rd s
+
I_{2}^{n}(t)
\lesssim
\int_{0}^{t}
\bE[|X_s - X_s^n|^p]
\rd s
+
([b]_{1} \vee |b(0)|)
n^{-p + \varepsilon p}.
\end{align*}
Therefore, by Gronwall's inequality, we obtain
\begin{align*}
\sup_{t \in [0, T]}
\bE[|X_t - X_t^n|^p]
\lesssim
n^{-p + \varepsilon p}.
\end{align*}

Now we assume $\alpha \in (0,1)$.
Let $\lambda > 0$ and let $u_{\lambda}=(u_{\lambda,1},\ldots,u_{\lambda,d})^{\top} \in C^{2 + \alpha/2}(\bR^d; \bR^d)$ be a solution to the following elliptic equation
\begin{equation*}
\lambda u_{\lambda} -\sL u_{\lambda} = b.
\end{equation*}
Then by Lemma \ref{lem_pde_01}, $u_{\lambda}$ satisfies
\begin{equation}\label{est_norm_u}
\|u_{\lambda}\|_{2 + \alpha/2}
\leq
C(\lambda)\|b\|_{\alpha},
\end{equation}
with $\lim_{\lambda \to \infty} C(\lambda) = 0$.
By It\^o's formula, we have for each $i=1,\ldots,d$,
\begin{align*}
\int_0^t
b_i(X_r)
\rd r 
&=
u_{\lambda, i}(x_{0})
-
u_{\lambda, i}(X_t)
+
\lambda
\int_0^t
u_{\lambda, i}(X_r)
\rd r
+
\int_0^t
\langle (\nabla u_{\lambda, i})(X_r),\sigma(X_r)\rd B_r\rangle,
\\
\int_0^t
b_i(X_r^n)
\rd r
&=
u_{\lambda, i}(x_{0})
-
u_{\lambda, i}(X_t^n)
+
\lambda
\int_0^t
u_{\lambda, i}(X_r^n)
\rd r
+
\int_0^t
\langle b(X_{\kappa_n(r)}^n) - b(X_r^n),(\nabla u_{\lambda, i})(X_r^n) \rangle
\rd r\\
&\quad
+
\frac{1}{2}
\int_0^t
\mathrm{Tr}
\big(
\big\{a(X_{\kappa_n(r)}^n) - a(X_r^n)\big\}
(\nabla^2 u_{\lambda, i})(X_r^n)
\big)
\rd r
\\&\quad+
\int_0^t
\langle (\nabla u_{\lambda, i})(X_r^n),\sigma(X_{\kappa_n(r)}^n)\rd B_{r} \rangle.
\end{align*}
Hence we have
\begin{align*}
I_{1}^{n}(t)
&\lesssim
\bE[|u_{\lambda}(X_t) - u_{\lambda}(X_t^n)|^p]
+
\lambda^p
\int_0^t
\bE[|u_{\lambda}(X_r) - u_{\lambda}(X_r^n)|^{p}]
\rd r\\
&\qquad
+
\sum_{i = 1}^{d}
\bE\Big[\Big|
\int_0^t
\Big\{
(\nabla u_{\lambda, i})(X_r)\sigma(X_r) - (\nabla u_{\lambda, i})(X_r^n)\sigma(X_{\kappa_n(r)}^n)
\Big\}
\rd B_r
\Big|^p\Big]\\
&\qquad
+
\sum_{i = 1}^d
\bE\Big[\Big|
\int_0^t
\langle b(X_r^n) - b(X_{\kappa_n(r)}^n),\nabla u_{\lambda, i}(X_r^n)\rangle
\rd r
\Big|^p
\Big]\\
&\qquad
+
\sum_{i,k, \ell = 1}^d
\bE\Big[\Big|
\int_0^t
\Big\{a_{k\ell}(X_{\kappa_n(r)}^n) - a_{k\ell}(X_r^n)\Big\}
(\partial_{k\ell}^{2}u_{\lambda, i})(X_r^n)
\rd r
\Big|^p
\Big]
\\
&=:
\sum_{j = 1}^{5}
I_{1,j}^{n}(t).
\end{align*}
By using \eqref{est_norm_u}, we have
\begin{equation}\label{I_1_1}
I_{1,1}^{n}(t)
\leq
\|\nabla u_{\lambda}\|_\infty^p
\bE[|X_t - X_t^n|^p]
\leq
C(\lambda)^p
\|b\|_{\alpha}^{p}
\bE[|X_t - X_t^n|^p].
\end{equation}
By the global Lipschitz continuity of $u_{\lambda}$,
\begin{equation}\label{I_1_2}
I_{1,2}^{n}(t)
\leq
\lambda^p
\|\nabla u_{\lambda}\|_\infty^p
\int_0^t
\bE[|X_r - X_r^n|^p]
\rd r.
\end{equation}
By Burkholder--Davis--Gundy inequality and Lemma \ref{lem_dif_em},
\begin{align}
I_{1,3}^{n}(t)
&\lesssim
\sum_{i = 1}^d
\Big\{
\bE\Big[\Big|
\int_0^t
\Big\{
(\nabla u_{\lambda, i})(X_r)\sigma(X_r) - (\nabla u_{\lambda, i})(X_r^n)\sigma(X_r^n)
\Big\}
\rd B_r\Big|^p\Big]
\notag
\\&
\hspace{2cm}+
\bE\Big[\Big|
\int_0^t
\Big\{
(\nabla u_{\lambda, i})(X_r^n)(\sigma(X_r^n) - \sigma(X_{\kappa_n(r)}^n)
\Big\}
\rd B_r
\Big|^p\Big]
\Big\}
\notag
\\&\lesssim
\|\nabla u_{\lambda}\|_{C^1_b}^p
\|\sigma\|_{C^1_b}^p
\int_0^t
\bE[|X_r - X_r^n|^p]\rd r
+
\|\nabla u_{\lambda}\|_\infty^p
\|\nabla \sigma\|_\infty^p
\int_0^t
\bE[|X_r^n - X_{\kappa_n(r)}^n|^p]
\rd r
\notag
\\&\lesssim
\|\nabla u_{\lambda}\|_{C^1_b}^p
\|\sigma\|_{C^1_b}^p
\int_0^t
\bE[|X_r - X_r^n|^p]\rd r
+
\|\nabla u_{\lambda}\|_\infty^p
\|\nabla \sigma\|_\infty^p
n^{-p/2}.
\label{I_1_3}
\end{align}
In particular, if $\sigma$ is a constant matrix, then since $\nabla \sigma=0$, we have
\begin{equation}\label{I_1_3_additive}
I_{1,3}^{n}(t)
\lesssim
\|\nabla u_{\lambda}\|_{C^1_b}^p
\int_0^t
\bE[|X_r - X_r^n|^p]\rd r.
\end{equation}
Since $\nabla u_{\lambda}$ is a bounded and globally Lipschitz continuous function, by using Proposition \ref{lem_quadratic_em} with $g=\nabla u_{\lambda}$, we have
\begin{equation}\label{I_1_4}
I_{1,4}^{n}(t)
\lesssim
([b]_{\alpha} \vee |b(0)|)(\|\nabla u_{\lambda}\|_\infty + [\nabla u_{\lambda}]_{\mathrm{Lip}})
n^{-\frac{1 + \alpha}{2}p + \varepsilon p}.
\end{equation}
Since $a \in C_b^3(\bR^d)$ and $\nabla^2 u_{\lambda}$ is bounded, by using Lemma \ref{lem_dif_em}, we have
\begin{equation}\label{I_1_5}
I_{1,5}^{n}(t)
\lesssim
\|\nabla a\|_\infty^p
\|\nabla^2 u_{\lambda}\|_\infty^p
\int_0^t
\bE[|X_{\kappa_n(r)}^n - X_r^n|^p]\rd r
\lesssim
\|\nabla a\|_\infty^p
n^{-p/2}.
\end{equation}
In particular, if $\sigma$ is a constant matrix, then since $\nabla a=0$, we have
\begin{equation}\label{I_1_5_additive}
I_{1,5}^{n}(t)
=0.
\end{equation}
By Proposition \ref{lem_quadratic_em} with $g=1$,
\begin{equation}\label{I_2}
I_{2}^{n}(t)
\lesssim
([b]_{\alpha} \vee |b(0)|)
n^{-\frac{1 + \alpha}{2}p + \varepsilon p}.
\end{equation}
By Burkholder--Davis--Gundy's inequality and Lemma \ref{lem_dif_em}, we have
\begin{align}
I_{3}^{n}(t)
+
I_{4}^{n}(t)
&\lesssim
\int_0^t
\bE[|\sigma(X_r) - \sigma(X_r^n)|^p]\rd r
+
\int_0^t
\bE[|\sigma(X_r^n) - \sigma(X_{\kappa_n(r)}^n)|^p]
\rd r
\notag
\\&\lesssim
\|\nabla \sigma\|_\infty^p
\Big\{
\int_0^t
\bE[|X_r - X_r^n|^p]
\rd r
+
\int_0^t
\bE[|X_r^n - X_{\kappa_n(r)}^n|^p]
\rd r
\Big\}
\notag
\\&\lesssim
\|\nabla \sigma\|_\infty^p
\Big\{
\int_0^t\bE[|X_r - X_r^n|^p]\rd r
+
n^{-p/2}
\Big\}.\label{I_3I_4}
\end{align}
In particular, if $\sigma$ is a constant matrix, then we have
\begin{equation}\label{I_345_additive}
I_{3}^{n}(t)=I_{4}^{n}(t)=0.
\end{equation}

If the diffusion coefficient $\sigma$ is not a constant matrix, then by combining \eqref{est_norm_u}, \eqref{I_1_1}, \eqref{I_1_2}, \eqref{I_1_3},
\eqref{I_1_4},  \eqref{I_1_5}, \eqref{I_2} and \eqref{I_3I_4}, there exist $C_{p}>0$ and $C_{p,\lambda}$ such that
\begin{align*}
\bE[|X_t - X_t^n|^p]
\leq
C_{p}
C(\lambda)^p\|b\|_\alpha^p
\bE[|X_t - X_t^n|^p]
+
C_{p,\lambda}
\Big\{
n^{-\frac{1 + \alpha}{2}p+ \varepsilon p}
+
n^{-\frac{p}{2}}
+
\int_0^t \bE[|X_r - X_r^n|^p]\rd r
\Big\}.
\end{align*}
Note that since $\varepsilon < \alpha / 2$, we have $-\frac{1 + \alpha}{2}p + \varepsilon p < -\frac{p}{2}$.
Therefore, by choosing $\lambda>\lambda_{0}$ with
\begin{align*}
C_{p}C(\lambda)^p\norm{b}_\alpha^p < 1
\end{align*}
and then applying Gronwall inequality, we obtain
\begin{align*}
\sup_{t \in [0, T]}
\bE[|X_t - X_t^n|^p]
\lesssim
n^{-p/2}.
\end{align*}

If the diffusion coefficient $\sigma$ is a constant matrix, then by \eqref{I_1_3_additive}, \eqref{I_1_5_additive} and \eqref{I_345_additive},
\begin{align*}
\bE[|X_t - X_t^n|^p]
\leq
C_{p}
C(\lambda)^p\|b\|_\alpha^p
\bE[|X_t - X_t^n|^p]
+
C_{p,\lambda}
\Big\{
n^{-\frac{1 + \alpha}{2}p+ \varepsilon p}
+
\int_0^t \bE[|X_r - X_r^n|^p]\rd r
\Big\}.
\end{align*}
Therefore, by the same argument as above, we obtain
\begin{align*}
\sup_{t \in [0, T]}
\bE[|X_t - X_t^n|^p]
\lesssim
n^{-\frac{1 + \alpha}{2}p+ \varepsilon p}.
\end{align*}
This concludes the proof.
\end{proof}

\subsection*{Acknowledgements}
The second author was supported by JSPS KAKENHI Grant Number 21H00988 and 23K12988.

\end{document}